\documentclass[12pt]{amsart}
\usepackage{a4}
\usepackage{amsthm, amsfonts, amssymb,latexsym}
\usepackage{enumerate, color}
\usepackage[english]{babel}
\usepackage[latin1]{inputenc}
\usepackage{tikz}
\usepackage{float}

\newtheorem{lemma}{Lemma}[section]

\newtheorem{theorem}{Theorem}[section]

\newcommand{\A}{\mathcal{A}}

\newcommand{\la}{\lambda}

\numberwithin{equation}{section}

\newcommand{\beq}[1]{\begin{equation}\label{#1}}
\newcommand{\eeq}{\end{equation}}

\title[New bounds for $(A+A)/(A+A)$]{If $(A+A)/(A+A)$ is small then the ratio set is large}

\author[ O. Roche-Newton]{Oliver Roche-Newton}

\address{O. Roche-Newton: School of Mathematics and Statistics, Wuhan University, Wuhan, Hubei Province, P.R.China. 430072 }
\email{o.rochenewton@gmail.com }

\begin{document}

\begin{abstract}
In this paper, we consider the sum-product problem of obtaining lower bounds for the size of the set
$$\frac{A+A}{A+A}:=\left \{ \frac{a+b}{c+d} : a,b,c,d \in A, c+d \neq 0 \right\},$$
for an arbitrary finite set $A$ of real numbers. The main result is the bound

$$\left| \frac{A+A}{A+A} \right| \gg \frac{|A|^{2+\frac{2}{25}}}{|A:A|^{\frac{1}{25}}\log |A|},$$
where $A:A$ denotes the ratio set of $A$. This improves on a result of Balog and the author \cite{BORN}, provided that the size of the ratio set is subquadratic in $|A|$. That is, we establish that the inequality
$$\left| \frac{A+A}{A+A} \right| \ll |A|^{2}  \Rightarrow |A:A| \gg \frac{ |A|^2}{\log^{25}|A|} .  $$
This extremal result answers a question similar to some conjectures in a recent paper of the author and Zhelezov \cite{ORNZ}.
\end{abstract} 
\maketitle
\section{Introduction}

Given a set $A$, we define its \textit{sum set} to be the set
$$A+A:=\{a+b:a,b \in A\}$$
and its \textit{product set} to be
$$AA:=\{ab:a,b \in A\}.$$
It was conjectured by Erd\H{o}s and Szemer\'{e}di \cite{ES} that, for any finite set of integers $A$, at least one of these two sets has near-quadratic growth. Solymosi \cite{solymosi} used a beautiful and elementary geometric argument to prove that for any finite set $A \subset \mathbb R$,
\begin{equation}
\max \{|A+A|,|AA|\} \gg \frac{|A|^{4/3}}{\log^{1/3}|A|}.
\label{soly}
\end{equation}
Recently, a breakthrough for this problem was achieved by Konyagin and Shkredov \cite{KS}. They adopted and refined the approach of Solymosi, whilst also utilising several other tools from additive combinatorics (such as the Balog-Szemer\'{e}di-Gowers Theorem) in order to prove that
\begin{equation}
\max \{|A+A|,|AA|\} \gg |A|^{\frac{4}{3}+\frac{1}{20598}}.
\label{KS}
\end{equation}
See \cite{KS} and the references contained therein for more background on the sum-product problem.

In this paper, we consider the closely related problem of establishing lower bounds for the set
$$ \frac{A+A}{A+A} := \left \{\frac{a+b}{c+d} : a,b,c,d \in A, c+d \neq 0 \right \}.$$
In \cite{BORN}, it was proven that, if $A$ is a set of strictly positive reals, then
\begin{equation}
\left| \frac{A+A}{A+A} \right| \geq 2|A|^2 -1,
\label{antal}
\end{equation}
thus establishing a sharp sum-product type estimate for this set defined by a combination of additive and multiplicative operations. The sharpness is illustrated by taking $A= \{1,2,\dots,N\}$. By taking $N$ arbitrarily large we see that the exponent $2$ cannot be improved, and in fact, for some smaller examples such as $N=3$, the bound in \eqref{antal} is completely sharp.

A question that remains about this set is that of determining for which sets $A$ the bound \eqref{antal} is sharp, or close to sharp. Note that if $|A+A| \ll |A|$ then it follows that
$$\left| \frac{A+A}{A+A} \right | \ll |A|^2,$$
and so the bound in \eqref{antal} is essentially sharp in such a case. However, we are not aware of other constructions which exhibit the sharpness of  the bound in question, and this motivated the conjecture in \cite{ORNZ} that
\begin{equation}
\left| \frac{A+A}{A+A} \right | \ll |A|^2 \Rightarrow |A+A| \ll |A|.
\label{dmitry}
\end{equation}
Some first weak results towards this conjecture were also proven in \cite{ORNZ}, establishing that if $|\frac{A+A}{A+A}| \ll |A|^2$ then the size of the sum set $A+A$ is sub-quadratic in $|A|$.

In this paper, we revisit the geometric approach of \cite{BORN}, based upon the beautiful work of Solymosi \cite{solymosi}, whilst also utilising the refined approach introduced in \cite{KS}, in order to establish a bound which improves upon \eqref{antal} as long as the ratio set\footnote{The \textit{ratio set} of $A$ is the set $A:A:=\{a/b:a,b \in A, b\neq 0\}$.} is not of its maximum possible size. The main new result is the following:

\begin{theorem} \label{thm:main}
Let $A$ be a finite set of real numbers. Then
\begin{equation}
\left| \frac{A+A}{A+A} \right| \gg \frac{|A|^{2+\frac{2}{25}}}{|A:A|^{\frac{1}{25}}\log |A|}.
\label{main1}
\end{equation}
In particular, it follows that,
\begin{equation}
\left| \frac{A+A}{A+A} \right | \ll |A|^2 \Rightarrow |A:A| \gg \frac{|A|^2}{ \log^{25}|A|}.
\label{main2}
\end{equation}
\end{theorem}

So, we obtain almost complete information about the size of the ratio set of a set $A$ for which $\left|\frac{A+A}{A+A}\right| \ll |A|^2$. The statement \eqref{main2} is similar in spirit to that of conjecture \eqref{dmitry}. However, note that \eqref{main2} is actually weaker than \eqref{dmitry}; indeed, if \eqref{dmitry} were proven to be true, it would automatically follow that \eqref{main2} is also true, since it is known that any finite set $A \subset \mathbb R$ for which $|A+A| \ll |A|$ satisfies $|A:A| \gg |A|^2$ (see, for example, the work of Li and Shen \cite{LS}).

Throughout the paper, the standard notation
$\ll,\gg$ is applied to positive quantities in the usual way. Saying $X\gg Y$ means that $X\geq cY$, for some absolute constant $c>0$. All logarithms are base $2$.

\section{Sketch of the proof}

The proof of Theorem \ref{thm:main} builds on the simple proof of \eqref{antal} from \cite{BORN} by incorporating a key new idea from \cite{KS}. Let us first briefly recap how the proof of \eqref{antal} goes.

Given a point $p \in (\mathbb R \setminus \{0\}) \times \mathbb R$, we use the notation $R(p)$ for the slope of the line connecting $p$ to the origin, so $R((x,y))=\frac{y}{x}$. For a point set $P$ in the plane, let $R(P)$ be the set of slopes of lines through the origin covering $P$; that is,
$$R(P):=\left\{R(p): p \in P \right \}.$$
Since $(A \times A) + (A \times A)=(A+A) \times (A+A)$, we have $\frac{A+A}{A+A}=R((A\times A)+(A \times A))$. Therefore, we seek to show that the sum set $(A \times A)+(A \times A)$ determines many lines through the origin. 

We first label the slopes needed to cover $A \times A$ in ascending order as $\la_1,\la_2,\dots,\la_{|A:A|}$. Let $l_i$ denote the line through the origin with slope $\la_i$. A simple geometric observation is that, if we fix a point $p_i$ on $l_i \cap (A \times A)$ and consider the set of sums
$$\{p_i+q: q \in (A \times A) \cap l_{i+1}\},$$
then the $|(A \times A) \cap l_{i+1}|$ points in this set all determine different slopes to the origin. Furthermore, these vector sums all lie strictly in between $l_i$ and $l_{i+1}$, which allows us to count elements of $\frac{A+A}{A+A}$ by summing over all $i$ without overcounting. Therefore, summing over all $1\leq i \leq |A:A|-1$, we obtain
$$\left| \frac{A+A}{A+A} \right| \geq \sum_{i=1}^{|A:A|-1} |l_{i+1} \cap (A \times A)|=|A|^2-1,$$
and a small refinement of this argument gives a better multiplicative constant, as in \eqref{antal}.

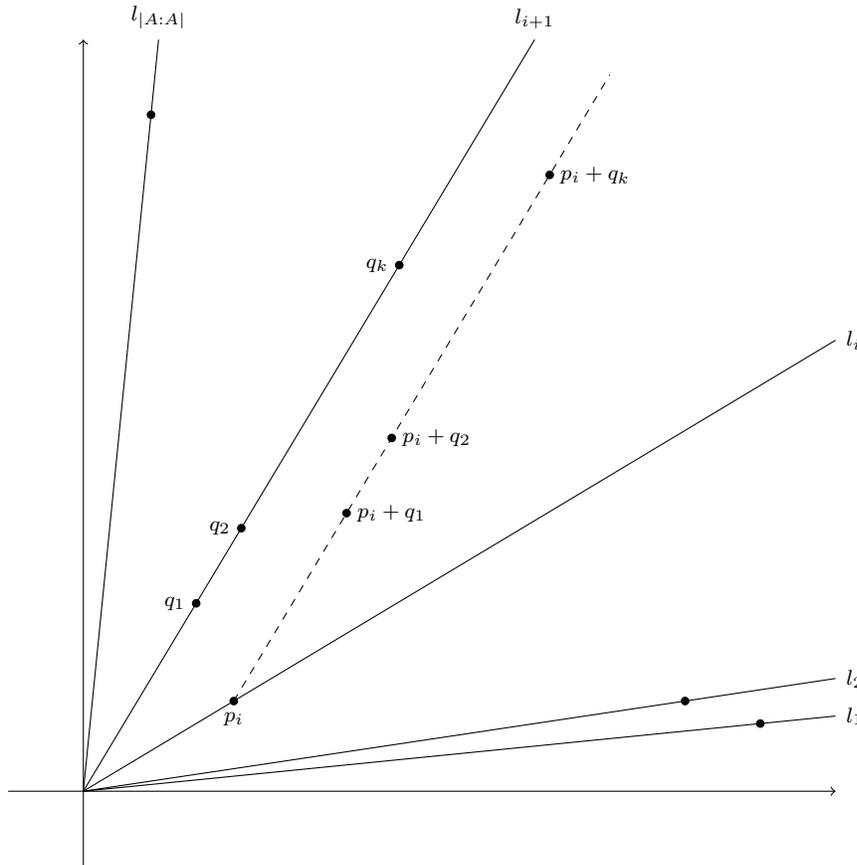
\begin{figure}[H]
\centering
\begin{tikzpicture}[font=\tiny, scale=1]

\draw [->] (0,-1) -- (0,10);
\draw [->] (-1,0) -- (10,0);
\draw (0,0) -- (10,1);
\draw (0,0) -- (10,1.5);
\draw (0,0) -- (10,6);
\draw (0,0) -- (6,10);
\draw (0,0) -- (1,10);
\draw [dashed] (2,1.2) -- (7,9.53333);
\draw[fill] (9,0.9) circle [radius=0.05];
\draw[fill] (8,1.2) circle [radius=0.05];
\draw[fill] (2,1.2) circle [radius=0.05];
\draw[fill] (1.5,2.5) circle [radius=0.05];
\draw[fill] (2.1,3.5) circle [radius=0.05];
\draw[fill] (4.2,7) circle [radius=0.05];
\draw[fill] (0.9,9) circle [radius=0.05];
\draw[fill] (3.5,3.7) circle [radius=0.05];
\draw[fill] (4.1,4.7) circle [radius=0.05];
\draw[fill] (6.2,8.2) circle [radius=0.05];
\node[right] at (10,1) {$l_1$};
\node[right] at (10,1.5) {$l_2$};
\node[right] at (10,6) {$l_{i}$};
\node[above] at (6,10) {$l_{i+1}$};
\node[above] at (1,10) {$l_{|A:A|}$};
\node[below] at (2,1.2) {$p_i$};
\node[left] at (1.5,2.5) {$q_1$};
\node[left] at (2.1,3.5) {$q_2$};
\node[left] at (4.2,7) {$q_{k}$};
\node[right] at (3.5,3.7) {$p_i+q_1$};
\node[right] at (4.1,4.7) {$p_i+q_2$};
\node[right] at (6.2,8.2) {$p_i+q_k$};
\end{tikzpicture}

\caption{Illustration of the key idea in the proof of inequality \eqref{antal}. In this case, $k=|(A \times A) \cap l_{i+1}|$, and we observe that $R(p_i+q_1)<R(p_i+q_2)<\dots<R(p_i+q_k)$.}

\centering

\end{figure}

We obtain an improvement on this argument by considering sums along more than just neighbouring lines. Suppose for simplicity that all of our lines through the origin support exactly $\tau$ points from $A \times A$. For some parameter $M$ (which we'd like to be as large as possible), we split the lines through the origin into consecutive clusters, each consisting of $M$ lines. We obtain at least
\begin{equation}
\tau {M \choose 2} - \mathcal E
\label{sketchy}
\end{equation}
distinct elements of $\frac{A+A}{A+A}$ coming from pairs of lines in a given cluster. The first term comes from repeating the previous argument for each of the ${M \choose 2}$ pairs of lines in the cluster, whilst the error term $\mathcal E$ is inserted to account for points which were overcounted in the first term. The next task is to show that, for sufficiently small $M$, this error term is insignificant compared to the main term. To bound the error term, we restate this quantity in terms of incidence geometry and utilise a variation on the Szemer\'{e}di-Trotter Theorem due to Pach and Sharir \cite{PS}.

Unfortunately, some extra work is needed to make this part of the argument work. In the proof of \eqref{antal}, the choice of the point $p_i$ to sum out from on each line can be made completely arbitrarily. However, after we convert the task of upper bounding $\mathcal E$ into an incidence problem, we see that the number of incidences could potentially be too large, depending on how we make the initial choice of a point from each line. Tools from extremal graph theory are utilised in order to show that it is possible to choose the points $p_i$ in such a way as to keep the error term small enough for the purposes of the argument.

There is another significant obstacle which this sketch has not yet mentioned. Roughly speaking, there are a comparatively small number of terms in the definition of $\mathcal E$ which we cannot deal with by the above methods. Therefore, we need to perform a non-trivial pigeonholing argument at the outset of the proof in order to find a large subset of the points in the plane which allow us to control these problematic terms. In order to do this, we set up another incidence problem, where the points in the configuration space become curves in the incidence space, and we choose a subset of the point set in the configuration space corresponding to poor curves in the incidence picture. In order to make this work, we have to modify the aforementioned clustering argument, instead building clusters of size $M+N$, where $N$ is considerably smaller than $M$.

\section{Preliminary results}
As stated in the sketch proof, a key tool in this paper is incidence theory. We will need the following result, which is a special case of a more general result of Pach and Sharir \cite{PS}. The result gives a generalisation of the Szemer\'{e}di-Trotter Theorem to certain well-behaved families of curves.
\begin{lemma} \label{thm:PS}
Let $\mathcal L$ be a family of curves and let $ \mathcal P$ be a set of points in the plane such that 
\begin{enumerate}
\item any two distinct curves from $\mathcal L$ intersect in at most two points of $\mathcal P$
\item for any two distinct points $p,q \in \mathcal P$, there exist at most two curves from $\mathcal L$ which pass through both $p$ and $q$.
\end{enumerate}
Then, for any $k \geq 2$, the set $\mathcal L_k$ of $k$ rich curves\footnote{To be precise $\mathcal L_k:=\{l \in L : |\{p \in \mathcal P : p \in l\}|\geq k \}$.} satisfies the bound
$$|\mathcal L_k| \ll \frac{|\mathcal P|^2}{k^3}+\frac{|\mathcal P|}{k}.$$
\end{lemma}

We will need a result from extremal graph theory similar to Tur\'{a}n's Theorem, which states that sufficiently dense graphs always contain large cliques. First, we introduce some basic definitions.

Given an undirected\footnote{Whenever we speak of a graph in this paper, it is assumed to be undirected.} graph $G=(V,E)$ with vertex set $V$ and edge set $E$, an induced subgraph of $G$ is a graph formed by taking a subset $V' \subset V$ and deleting all edges which contain one or more vertices from $V \setminus V'$. For two subsets $V_1,V_2 \subset V$, the notation $E(V_1,V_2)$ is used for the set of edges connecting $V_1$ and $V_2$; that is
$$E(V_1,V_2):=\{\{v_1,v_2\} \in E:v_1 \in V_1, v_2 \in V_2\}.$$
The notation $K_r$ is used for the complete graph on $r$ vertices. A set of verices which form a complete induced subgraph is said to be a clique.

Given a graph $G=(V,E)$, a subset of the vertex set $V' \subset V$ is said to be an \textit{independent set} if there are no edges connecting two elements of $V'$. A graph $G=(V,E)$ is said to be \textit{$r$-partite} if it is possible to write $V$ as a disjoint union $V=V_1 \cup V_2\cup\dots \cup V_r$ with each of the $V_i$ being independent sets.

We will need the following version of the Lov\'{a}sz Local Lemma. This precise statement is Corollary 5.1.2 in \cite{AS}.

\begin{theorem} \label{LLL} [Lov\'{a}sz Local Lemma]
Let $A_1,A_2,\dots,A_n$ be events in an arbitrary probability space. Suppose that each event $A_i$ is mutually independent from all but at most $d$ of the events $A_j$ with $j\neq i$. Suppose also that the probability of the event $A_i$ occuring is at most $p$, for all $1\leq i \leq n$. Finally, suppose that
$$ep(d+1) \leq 1 .$$
Then, with positive probability, none of the events $A_1,\dots,A_n$ occur.\footnote{Using standard notation from probability theory, the conclusion is that $\Pr(\bar{A_1} \land \bar{A_2} \land \cdots \land \bar{A_n})>0$.}
\end{theorem}

This will now be used to derive the following variant of Tur\'{a}n's Theorem. The proof, due to a private communication with Noga Alon, is similar to that of Proposition 2.4 in \cite{alon}, or Proposition 5.3.3 in \cite{AS}. The lemma plays an important role in the proof of the main theorem of this paper.\footnote{In an earlier draft of this paper, a weaker version of Lemma \ref{EGT} was proven via an application of Tur\'{a}n's Theorem. The improvement to the present version is evident in the relaxation of the condition \eqref{condition}. This improved version of Lemma \ref{EGT} has resulted in a small quantitative improvement to the main result of this paper, in comparison to the earlier draft. I am very grateful to Noga Alon for his generous help.}


\begin{lemma} \label{EGT}
Let $G$ be an $r$-partite graph on $rk$ vertices consisting of $r$ independent sets $V_1,V_2,\dots,V_r$ of vertices, each with cardinality $k$. Suppose that, for each $1 \leq i<j \leq r$ we have
\begin{equation}
|E(V_i,V_j)| > k^2\left(1-\frac{1}{e(2r-3)}\right).
\label{condition}
\end{equation}
Then $G$ contains $K_r$ as an induced subgraph.
\end{lemma}

\begin{proof}
Let $K=\{v_1,v_2,\dots,v_r\}$ be a random set of $r$ vertices, where each $v_i \in V_i$ is chosen uniformly and independently among all vertices in $V_i$. The aim is to show that, with positive probability, $K$ forms a clique. It then follows immediately that $G$ contains $K_r$ as an induced subgraph.

For each $1 \leq i<j \leq r$, let $E_{ij}$ be the event that $v_i$ and $v_j$ are not connected. It follows from condition \eqref{condition} that $\Pr(E_{ij}) \leq \frac{1}{e(2r-3)}$. Furthermore, the events $E_{ij}$ and $E_{i'j'}$ are mutually independent, provided that $\{i,j\} \cap\{i',j'\} = \emptyset$. This means that the event $E_{ij}$ is mutually independent from all but at most $2r-4$ of the other events $E_{i'j'}$.

We now apply Theorem \ref{LLL} with $d=2r-4$ and $p=\frac{1}{e(2r-3)}$. Note that
$$ep(d+1)=e\frac{1}{e(2r-3)}(2r-3)\leq 1.$$
Therefore, with positive probability, none of the events $E_{ij}$ occur. That is, with positive probability, the random set of vertices $K$ forms a clique.
\end{proof}

The following example shows that the lemma does not hold if \eqref{condition} is replaced by the bound
$$E(V_i,V_j) \geq k^2\left(1-\frac{1}{r-1}\right).$$
In other words, Lemma \ref{EGT} is tight up to determining the correct value of the multiplicative constant in the denominator. Indeed, let $G$ be the complete $r$-partite graph with $k$ vertices in each of the independent sets $V_1,\dots,V_r$. For simplicity, let us assume that $r-1$ divides $k$. Write $V_1=\{v_1,\dots,v_k\}$ and delete all of the edges between $v_1,\dots,v_{\frac{k}{r-1}}$ and the vertices of $V_2$, removing a total of $\frac{k^2}{r-1}$ edges between $V_1$ and $V_2$. Then, delete all of the edges between $v_{\frac{k}{r-1}+1},\dots,v_{2\frac{k}{r-1}}$ and $V_3$, again removing a total of $\frac{k^2}{r-1}$ edges between $V_1$ and $V_3$.

Repeat this process for each of the independent sets $V_4,\dots,V_r$ and label the resulting graph $G'$. We have at least
$$k^2-\frac{k^2}{r-1}=k^2\left(1-\frac{1}{r-1}\right)$$
edges between each of the independent sets of vertices of $G'$. However, the graph does not contain $K_r$ as an induced subgraph, since any such subgraph must contain a vertex from $V_1$, and this vertex has no edges between it and one of the other independent sets $V_2,\dots , V_r$.

It may be an interesting question in its own right to determine the correct multiplicative constants for a completely optimal version of Lemma \ref{EGT}.

\section{Proof of Theorem \ref{thm:main}}

Consider the point set $A\times A$ in the plane. Without loss of generality, we may assume that $A$ consists of strictly positive reals, and so this point set lies exclusively in the positive quadrant. The aim is to eventually show that $R((A \times A)+(A\times A))$ is large. For $\la \in A:A$, let $\mathcal A_{\la}$ denote the set of points on the line through the origin with slope $\la$ and let $A_{\la}$ denote the projection of this set onto the horizontal axis. That is,
$$\mathcal A_{\la}:=\{(x,y) \in A \times A:y=\la x\},\,\,\,\,\,\,\,\,A_{\la}:=\{x:(x,y) \in \mathcal A_{\la}\}.$$
Note that $|\mathcal A_{\la}|=|A_{\la}|$ and
$$\sum_{\la \in A:A} |A_{\la}|=|A|^2.$$

We begin by dyadically decomposing this sum and applying the pigeonhole principle in order to find a large subset of $A\times A$ consisting of points which lie on lines of similar richness. Note that
$$\sum_{\la:|A_{\la}| \leq \frac{|A|^2}{2|A:A|}} |A_{\la}| \leq \frac{|A|^2}{2},$$ 
and so
$$\sum_{\la:|A_{\la}| \geq \frac{|A|^2}{2|A:A|}} |A_{\la}| \geq \frac{|A|^2}{2}.$$
Dyadically decompose the sum to get
$$\sum_{j \geq 1}^{\lceil \log |A| \rceil} \sum_{\la: 2^{j-1}\frac{|A|^2}{2|A:A|} \leq |A_{\la}| < 2^j\frac{|A|^2}{2|A:A|}}|A_{\la}| \geq \frac{|A|^2}{2}.$$
Therefore, there exists some $\tau \geq \frac{|A|^2}{2|A:A|}$ such that
\begin{equation}
\tau|S_{\tau}| \gg \sum_{\la \in S_{\tau}} |A_{\la}| \gg \frac{|A|^2}{\log|A|},
\label{Sbound}
\end{equation}
where $S_{\tau}:=\{ \la : \tau \leq |A_{\la}| <2\tau \}$. Since $\tau$ is fixed throughout the proof, we simply write $S=S_{\tau}$.




Let $M$ and $N$ be integer parameters. We define
$$M=\lfloor c_M \tau^{1/5} \rfloor \,\,\,\,\,\text{and} \,\,\,\,\ N=\lfloor c_N \tau^{1/25} \rfloor,$$
where $c_M$ and $c_N$ are sufficiently small absolute constants. The constant $c_M$ depends on other absolute constants, while $c_N$ depends on $c_M$ along with other absolute constants. At various points in the proof, we will need to ensure that the choices of $c_M$ and $c_N$ are made in such a way as to ensure that certain bounds hold. We will keep a record of these requirements, and devote some time at the conclusion of the proof to ensuring that suitable choices of $c_M$ and $c_N$ are indeed possible to make. The first requirement is that
\begin{equation}
1 \leq N \leq M \leq \frac{|S|}{2}.
\label{req1}
\end{equation}

Following the approach of Konyagin and Shkredov \cite{KS} with a small adaptation, we split the slopes of our lines through the origin into clusters of size $M+N \leq 2M$. To be precise, if we list the elements of $S$ in increasing order, the first cluster $S_1$ consists of the first $M+N$ elements in this list, and so on for the j'th cluster $S_j$. We repeat this process until there are strictly less than $M+N$ slopes unaccounted for. Note that there are exactly $\left \lfloor \frac{|S|}{M+N} \right \rfloor$ clusters.

Now, let $S_j$ be an arbitrary cluster and let $\la_{min}$ and $\la_{max}$ be its smallest and largest elements respectively. Following the observations from papers \cite{solymosi} and \cite{BORN}, note that
$$R_j:=\{R(p+q): p,q \in \bigcup_{\la \in S_j}\mathcal A_{\la}\} \subset [\la_{min},\la_{max}].$$
Therefore, for any $1 \leq j,j' \leq \lfloor \frac{|S|}{M+N} \rfloor$ with $j \neq j'$, the sets $R_j$ and $R_{j'}$ are disjoint, and we can sum over the clusters without overcounting. That is, we know that
\begin{equation}
\left| \frac{A+A}{A+A} \right| \geq \left| \bigcup_j R_j \right| = \sum_j |R_j|.
\label{plan}
\end{equation}
Given this information, we now seek to show that $|R_j|$ is large for all $j$.

We arbitrarily divide the $M+N$ elements of $S_j$ into two distinct sets $T_j$ and $U_j$,  such that $|T_j|=M$ and $|U_j|=N$. Although this decomposition is completely arbitrary, let us take $T_j$ to be the $M$ smallest elements of $S_j$, and let $U_j$ be the remaining $N$ elements.

For technical reasons which will become clearer later in the proof, we need to find a large subset $P_j \subset \cup_{\la \in T_j} \mathcal A_{\la}$ such that for all $(a,\la a) \in P_j$ and every $(\la_{1},\la_{2}) \in U_j \times U_j$ with $\la_{1} \neq \la_{2}$,
\begin{equation}
|\{(x,y) \in A_{\la_1} \times A_{\la_2} : R((a,\la a)+(x,\la_1 x))=R((a,\la a)+(y,\la_2 y))\}| \leq \tau^{1-\frac{1}{25}}.
\label{technical}
\end{equation}

To find this subset we make use of some incidence theory. First, take an arbitrary pair $(\la_1,\la_2) \in U_j \times U_j$ with $\la_1 \neq \la_2$. For each $(a,\la a) \in \cup _{\la \in T_j} \mathcal A_{\la}$, define $l_{(a, \la a)}$ to be the curve
$$l_{(a,\la a)}:=\{(x,y) \in \mathbb R^2: (\la a+ \la_1x)(a+y)=(\la a+\la_2y)(a+x) \},$$
and let $\mathcal L$ be the set of curves
$$\mathcal L:=\{l_{(a,\la a)}: (a,\la a) \in \cup_{\la \in T_j} \mathcal A_{\la}\}.$$
Define $\mathcal P=A_{\la_1} \times A_{\la_2}$ and note that 
$$|l_{(a,\la a)} \cap P|=|\{(x,y) \in A_{\la_1} \times A_{\la_2}:R((a,\la a)+(x,\la_1 x))=R((a,\la a)+(y,\la_2 y))\}|.$$

We can check that this family of curves satisfy the conditions of Lemma \ref{thm:PS}. This check involves long and unenlightening calculations and so it is relegated to an appendix in an effort to avoid disrupting the flow of the proof. Therefore, if we define $\mathcal L_k$ to be the set of curves supporting at least $k$ points from $\mathcal P$, Lemma \ref{thm:PS} tells us that
$$|\mathcal L_k| \ll \frac{|\mathcal P|^2}{k^3}+\frac{|\mathcal P|}{k}.$$
Setting $k=\tau^{1-\frac{1}{25}}$ and recalling that $|\mathcal P|=|A_{\la_1}||A_{\la_2}| \ll \tau^2$, we have
$$|\mathcal L_k| \ll \tau^{1+\frac{3}{25}}.$$
That is, there are at most $C\tau^{1+\frac{3}{25}}$ elements $(a,\la a) \in \cup_{\la \in T_j} \mathcal A_{\la}$ such that
$$|\{(x,y) \in A_{\la_1} \times A_{\la_2} : R((a,\la a)+(x,\la_1 x))=R((a,\la a)+(y,\la_2 y))\}| \geq \tau^{1-\frac{1}{25}},$$
where $C$ is an absolute constant. 

We then repeat this process for each of the ${N \choose 2} \leq N^2$ possible choices for a distinct pair $(\la_1,\la_2) \in U_j \times U_j$. There are at most $CN^2 \tau^{1+\frac{3}{25}}$ elements $(a,\la a) \in \cup_{\la \in T_j} \mathcal A_{\la}$ such that
$$|\{(x,y) \in A_{\la_1} \times A_{\la_2} : R((a,\la a)+(x,\la_1 x))=R((a,\la a)+(y,\la_2 y))\}| \geq \tau^{1-\frac{1}{25}}$$
for some $(\la_1,\la_2) \in U_j \times U_j$ with $\la_1 \neq \la_2$. All of the other elements $(a,\la a) \in \cup_{\la \in T_j} \mathcal A_{\la}$ have the property that, for all $(\la_1,\la_2) \in U_j \times U_j$ with $\la_1 \neq \la_2$
\begin{equation}
|\{(x,y) \in A_{\la_1} \times A_{\la_2} : R((a,\la a)+(x,\la_1 x))=R((a,\la a)+(y,\la_2 y))\}| \leq \tau^{1-\frac{1}{25}}.
\label{wanted}
\end{equation}

Define $P_j \subset \cup_{\la \in T_j} \mathcal A_{\la}$ to be the set of elements of  $\cup_{\la \in T_j} \mathcal A_{\la}$ such that \eqref{wanted} holds for all $\la_1,\la_2 \in U_j$ with $\la_1 \neq \la_2$. The previous argument tells us that
\begin{equation}
|P_j| \geq \tau M - CN^2\tau^{1+\frac{3}{25}}.
\label{Pjbound}
\end{equation}
By choosing $c_N$ sufficiently small, it follows that
\begin{equation}
|P_j| \geq \frac{\tau M}{2}.
\label{Pjbound2}
\end{equation}
Indeed, we will check later that it is possible to choose $c_N$ and $c_M$ so that
\begin{equation}
c_N \leq \left(\frac{c_M}{4C}\right)^{1/2},
\label{req2}
\end{equation}
thus ensuring that \eqref{Pjbound2} holds.

We now have a set $P_j$ of at least $\frac{\tau M}{2}$ points. It will be computationally convenient for all of these points to lie on lines having exactly the same richness, and so yet another process of pigeonholing is needed.

Let $T_j' \subset T_j$ be the set
$$T_j':=\left\{\la \in T_j: |\mathcal A_{\la} \cap P_j| \geq \frac{\tau}{4} \right\}.$$
We have
\begin{align*}
\frac{\tau M}{2} &\leq \sum_{\la \in T_j} |\mathcal A_{\la} \cap P_j|
\\&=\sum_{\la \in T_j'} |\mathcal A_{\la} \cap P_j|+\sum_{\la \in T_j \setminus T_j'} |\mathcal A_{\la} \cap P_j|
\\& \leq \sum_{\la \in T_j'} |\mathcal A_{\la} \cap P_j|+\frac{\tau M}{4},
\end{align*}
and hence $\sum_{\la \in T_j'} |\mathcal A_{\la} \cap P_j| \geq \frac{\tau M}{4}$. Since $P_j \subset \cup_{\la \in T_j} \mathcal A_{\la}$, we have $|\A_{\la}\cap P_j| \leq 2\tau$ for all $\la \in T_j'$. It then follows that $|T_j'| \geq \frac{M}{8}$. Finally, let $T_j''$ be the subset of $T_j'$ consisting of the first (i.e. the shallowest) $\frac{M}{8}$ slopes in $T_j'$. For each $\la \in T_j''$, let $\mathcal A_{\la}''$ be the first (i.e. nearest the origin) $\frac{\tau}{4}$ points from $\mathcal A_{\la} \cap P_j$.\footnote{In order to simplify the increasingly complicated notation, we assume here that the values $\frac{M}{8}$ and $\frac{\tau}{4}$ are integers. Strictly speaking, we should have $|T_j''|=\lceil \frac{M}{8} \rceil$ and $|\mathcal A_{\la}''|=\lceil \frac{\tau}{4} \rceil.$} Define $P_j''$ to be the point set
$$P_j'':=\cup_{ \la \in T_j''} \A_{\la}''.$$
This is a point set of cardinality $\frac{M\tau}{32}$, consisting of $\frac{M}{8}$ lines through the origin, each supporting $\frac{\tau}{4}$ points. We will work with $P_j''$ for the remainder of the proof.

For each $\la \in T_j''$, we choose a representative point $(a_{\la},\la a_{\la}) \in \mathcal A_{\la}''$. This point plays the same role as the fixed point $p_i$ in the earlier sketch proof of \eqref{antal}. We label this set of representatives $\mathcal B_j$. The process of choosing this set of $\frac{M}{8}$ representatives so that is has certain important properties is rather a delicate one, and we will return to this choice later in the proof.

Define $R_j'$ to be the set
$$R_j':=\bigcup_{ (\la,\la') \in U_j \times T_j'' } R \big ( (a_{\la'}, \la' a_{\la'})+\A_{\la} \big ).$$
Since $R_j'\subset R_j$, it will be sufficient to show that $|R_j'|$ is large for all $j$, and then use \eqref{plan} to complete the argument.

Now, let $\la \in U_j$ and $\la' \in T_j''$. An important observation from \cite{BORN} is that the set
\begin{equation}
R((a_{\la'},\la'a_{\la'})+\mathcal A_{\la})
\label{realisation}
\end{equation}
has cardinality $|A_{\la}| \geq \tau$. To see this, label the elements of $A_{\la}$ in increasing order; so $a_1,\dots,a_{|A_{\la}|}$ are the elements of $A_{\la}$ and $a_1<a_2< \dots < a_{|A_{\la}|}$. Then
$$R((a_{\la'},\la' a_{\la'})+(a_1,\la a_1))<R((a_{\la'},\la'a_{\la'})+(a_2,\la a_2))< \dots<R((a_{\la'},\la' a_{\la'})+(a_{|A_{\la}|},\la a_{|A_{\la}|})).$$
A visual illustration of why these inequalities hold is given in the sketch proof of \eqref{antal}.

We would like to repeat this argument for each pair $(\la,\la') \in U_j \times T_j''$ in order to deduce that $|R_j'| \geq \sum_{(\la,\la') \in U_j \times T_j''} \tau \gg MN\tau$. However, it is possible that this is an overcount, since a single slope from $R_j'$ will be counted more than once in this process if it occurs as an element of 	
$$R((a_{\la_3},\la_3a_{\la_3})+\mathcal A_{\la_1}) \cap R((a_{\la_4},\la_4a_{\la_4})+\mathcal A_{\la_2}),$$ for some quadruple $(\la_1,\la_2,\la_3,\la_4) \in U_j \times U_j \times T_j'' \times T_j''$ with $(\la_1,\la_3) \neq (\la_2,\la_4)$.

For this reason, it is necessary to introduce an error term and deal with this overcounting. Some more notation is needed.

For each $z \in R_j'$, let $r(z)$ denote the number of sets of the form \eqref{realisation} that $z$ belongs to. To be precise
$$r(z)=|\{(\la, \la') \in U_j \times T_j'' : z \in R((a_{\la'},\la'a_{\la'})+\mathcal A_{\la})\}|.$$
Note that
$$|R_j'| \geq \sum _{(\la,\la') \in U_j \times T_j''} \tau - \sum_{z \in R_j'} (r(z)-1).$$
For all $z \in R_j'$, the bound $r(z)-1 \leq (r(z)-1)r(z)$ holds trivially. Therefore
\begin{align*}
|R_j'| &\geq \sum _{(\la,\la') \in U_j \times T_j''} \tau - \sum_{z \in R_j'} (r(z)-1)r(z)
\\&=\sum _{(\la,\la') \in U_j \times T_j''} \tau - \sum_{z \in R_j'} |\{(\la_1,\la_2,\la_3,\la_4) \in U_j \times U_j \times T_j'' \times T_j'' :
\\ & z \in  R((a_{\la_3},\la_3a_{\la_3})+\mathcal A_{\la_1}) \cap R((a_{\la_4},\la_4a_{\la_4})+\mathcal A_{\la_2}), (\la_1,\la_3) \neq (\la_2,\la_4) \}|
\end{align*}
A rearrangement of this inequality yields
\begin{align}
|R_j'| &\geq \sum_{(\la,\la') \in U_j \times T_j''} \tau - \sum_{(\la_1,\la_3), (\la_2,\la_4) \in U_j \times T_j'':(\la_1,\la_3) \neq (\la_2,\la_4)} \mathcal E(\la_1,\la_2,\la_3,\la_4)  
\\&= \frac{MN \tau}{8} - \sum_{(\la_1,\la_3), (\la_2,\la_4) \in U_j \times T_j'':(\la_1,\la_3) \neq (\la_2,\la_4)} \mathcal E(\la_1,\la_2,\la_3,\la_4), \label{key}
\end{align}
where
$$\mathcal E (\la_1,\la_2,\la_3,\la_4):=|\{ z \in R((a_{\la_3},\la_3a_{\la_3})+\mathcal A_{\la_1}) \cap R((a_{\la_4},\la_4a_{\la_4})+\mathcal A_{\la_2})\}|. $$

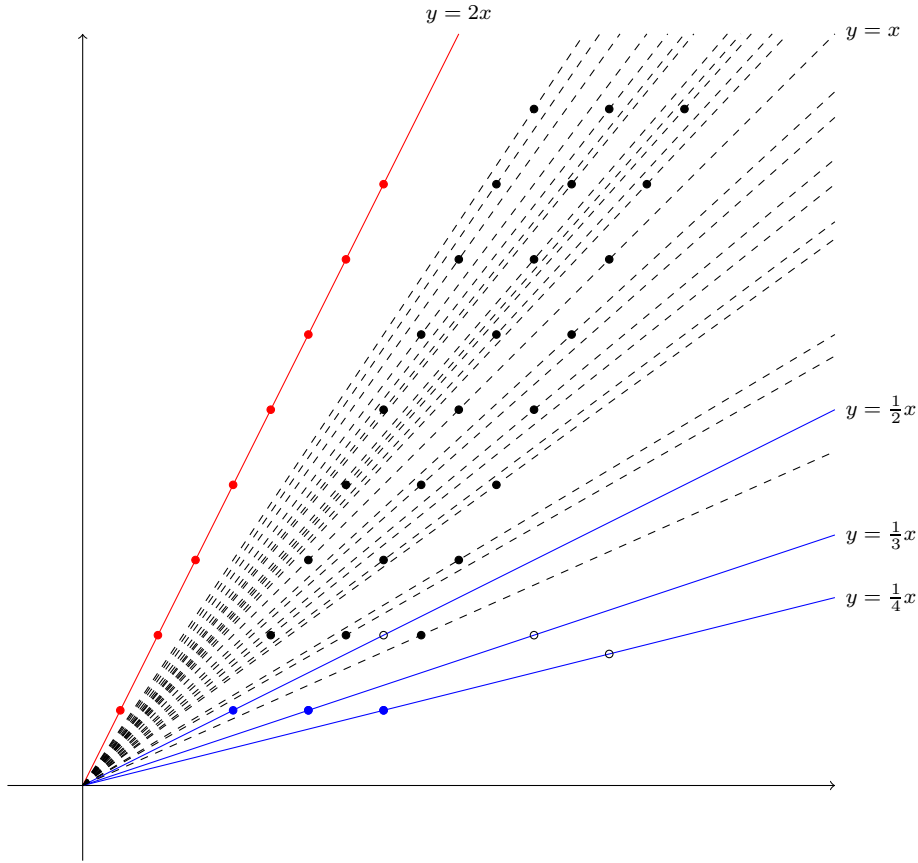
\begin{figure}[H]
\centering
\begin{tikzpicture}[font=\tiny, scale=1]

\draw [->] (0,-1) -- (0,10);
\draw [->] (-1,0) -- (10,0);
\draw [blue] (0,0) -- (10,2.5);
\draw [blue] (0,0) -- (10,3.333333);
\draw [blue] (0,0) -- (10,5);
\draw [red] (0,0) -- (5,10);
\draw [dashed] (0,0) -- (7.142857,10);
\draw [dashed] (0,0) -- (8,10);
\draw [dashed] (0,0) -- (8.571248,10);
\draw [dashed] (0,0) -- (10,10);
\draw [dashed] (0,0) -- (10,8.333333);
\draw [dashed] (0,0) -- (10,7.5);
\draw [dashed] (0,0) -- (10,6);
\draw [dashed] (0,0) -- (6.666666,10);
\draw [dashed] (0,0) -- (7.77777,10);
\draw [dashed] (0,0) -- (8.888888,10);
\draw [dashed] (0,0) -- (10,8);
\draw [dashed] (0,0) -- (10,5.714285);
\draw [dashed] (0,0) -- (10,4.444444);
\draw [dashed] (0,0) -- (8.75,10);
\draw [dashed] (0,0) -- (10,8.888888);
\draw [dashed] (0,0) -- (10,7.27272727);
\draw [dashed] (0,0) -- (7.5,10);
\draw [dashed] (0,0) -- (9.1666666,10);
\draw [dashed] (0,0) -- (10,9.230769);
\draw [dashed] (0,0) -- (6.875,10);
\draw [dashed] (0,0) -- (8.125,10);
\draw [dashed] (0,0) -- (9.375,10);
\draw[fill,blue] (2,1) circle [radius=0.05];
\draw[fill,blue] (3,1) circle [radius=0.05];
\draw[fill,blue] (4,1) circle [radius=0.05];
\draw (4,2) circle [radius=0.05];
\draw (6,2) circle [radius=0.05];
\draw (7,1.75) circle [radius=0.05];
\draw[fill,blue] (3,1) circle [radius=0.05];
\draw[fill,blue] (4,1) circle [radius=0.05];
\draw[fill,red] (1,2) circle [radius=0.05];
\draw[fill,red] (2,4) circle [radius=0.05];
\draw[fill,red] (3,6) circle [radius=0.05];
\draw[fill,red] (4,8) circle [radius=0.05];
\draw[fill,red] (1.5,3) circle [radius=0.05];
\draw[fill,red] (2.5,5) circle [radius=0.05];
\draw[fill,red] (3.5,7) circle [radius=0.05];
\draw[fill,red] (0.5,1) circle [radius=0.05];
\draw[fill] (2.5,2) circle [radius=0.05];
\draw[fill] (3.5,2) circle [radius=0.05];
\draw[fill] (4.5,2) circle [radius=0.05];
\draw[fill] (3.5,4) circle [radius=0.05];
\draw[fill] (4.5,4) circle [radius=0.05];
\draw[fill] (5.5,4) circle [radius=0.05];
\draw[fill] (4.5,6) circle [radius=0.05];
\draw[fill] (5.5,6) circle [radius=0.05];
\draw[fill] (6.5,6) circle [radius=0.05];
\draw[fill] (5.5,8) circle [radius=0.05];
\draw[fill] (6.5,8) circle [radius=0.05];
\draw[fill] (7.5,8) circle [radius=0.05];
\draw[fill] (3,3) circle [radius=0.05];
\draw[fill] (4,3) circle [radius=0.05];
\draw[fill] (5,3) circle [radius=0.05];
\draw[fill] (4,5) circle [radius=0.05];
\draw[fill] (5,5) circle [radius=0.05];
\draw[fill] (6,5) circle [radius=0.05];
\draw[fill] (5,7) circle [radius=0.05];
\draw[fill] (6,7) circle [radius=0.05];
\draw[fill] (7,7) circle [radius=0.05];
\draw[fill] (6,9) circle [radius=0.05];
\draw[fill] (7,9) circle [radius=0.05];
\draw[fill] (8,9) circle [radius=0.05];
\node[right] at (10,2.5) {$y=\frac{1}{4}x$};
\node[right] at (10,3.33333333) {$y=\frac{1}{3}x$};
\node[right] at (10,5) {$y=\frac{1}{2}x$};
\node[above] at (5,10) {$y=2x$};
\node[right] at (10,10) {$y=x$};
\end{tikzpicture}

\caption{An illustration of the geometric setup for the proof. In this representation $\tau=8$, $M=24$ and $N=1$; the three blue lines correspond to elements of $T_j''$ (recall that $|T_j''|=M/8=3$) and the single red line corresponds to the sole element of $U_j$. The black points are all elements of the sum set of the red and blue points. The dotted lines, which form a minimal covering of the black points by lines through the origin, are in one-to-one correspondence with the elements of $R_j'$. In this case, $|R_j'|=22$, which is smaller than its maximum possible value of $24$. This loss occurs because the dotted line with equation $y=x$ has three points on it (in the previously introduced notation, this means that $r(1)=3$). An error term is introduced to compensate for this overcounting of rich dotted lines. One could obtain a different version of this picture by making a different choice for the three representative (blue) points on the blue lines and considerable care is taken throughout the proof to make a suitable choice for the representative points in order to have as few rich dotted lines as possible.}

\centering

\end{figure}

Note that the value $\mathcal E (\la_1,\la_2,\la_3,\la_4)$ is dependent upon our earlier choice for the set of representatives $\mathcal B_j$. The following lemma tells us that it is possible to construct this set $\mathcal B_j$ in such a way as to ensure that $\mathcal E (\la_1,\la_2,\la_3,\la_4)$ is always small for $\la_3 \neq \la_4$.

\begin{lemma} \label{thm:hardwork} The set of representatives $\mathcal B_j$ can be chosen in such a way as to ensure that, for all $(\la_1,\la_3), (\la_2,\la_4) \in U_j \times T_j''$ such that  $\la_3 \neq \la_4$, we have
\begin{equation}
\mathcal E (\la_1,\la_2,\la_3,\la_4) \leq \tau^{19/25}.
\label{Ebound}
\end{equation}
\end{lemma}

\begin{proof} Fix two distinct slopes $\la_3$ and $\la_4$ in $T_j''$ and let $(a,\la_3 a) \in \mathcal A_{\la_3}$ and $(b,\la_4 b) \in \mathcal A_{\la_4}$. Fix $(\la_1,\la_2) \in U_j \times U_j$. We say that the pair $((a,\la_3 a),(b,\la_4 b))$ is \textit{good for $(\la_1,\la_2)$} if
$$|\{(x,y) \in A_{\la_1} \times A_{\la_2}: R((a,\la_3 a)+(x,\la_1 x))=R((b,\la_2 b)+(y,\la_4 y))\}| \leq \tau^{19/25}.$$
Note that this condition is equivalent to the bound
$$|\{(x,y) \in A_{\la_1} \times A_{\la_2}: (\la_3 a+\la_1 x)(b+y)=(\la_4 b+\la_2 y)(a+x)\}| \leq \tau^{19/25}.$$
We will use incidence theory to prove that most elements of $\mathcal A_{\la_3} \times \mathcal A_{\la_4}$ are good for $(\la_1,\la_2)$.

For $a \in A_{\la_3}$ and $b \in A_{\la_4}$, define $l_{a,b}$ to be the curve
$$l_{a,b}:=\{(x,y): (\la_3 a+\la_1 x)(b+y)=(\la_4 b+\la_2 y)(a+x)\}.$$
Define $\mathcal L'$ to be the set of curves
$$\mathcal L':=\{l_{a,b}: (a,b) \in A_{\la_3} \times A_{\la_4}\}$$
and let $\mathcal P$ be the point set $\mathcal P:=A_{\la_1} \times A_{\la_2}$. We can check that the family of curves $\mathcal L'$ satisfy the conditions of Lemma \ref{thm:PS}. Once again, this calculation is delayed until the appendix. Therefore, if we define $\mathcal L_k'$ to be the set of curves from $\mathcal L'$ which contain at least $k$ points from $\mathcal P$, 
we have
$$|\mathcal L_k'| \ll \frac{|\mathcal P|^2}{k^3}+\frac{|\mathcal P|}{k}.$$
Setting the parameter $k=\tau^{19/25}$ and recalling that $|A_{\la_1}|, |A_{\la_2}| \leq 2\tau$, we then have
$$|\mathcal L_k'| \ll \frac{\tau^4}{\tau^{57/25}}+\frac{\tau^2}{\tau^{19/25}}\ll \tau^{43/25}.$$
Therefore, for some absolute constant $C'$, there exist at most $C'\tau^{43/25}$ pairs $(a,b) \in A_{\la_3} \times A_{\la_4}$ such that
$$|\{(x,y) \in A_{\la_1} \times A_{\la_2}: (\la_3 a+\la_1 x)(b+y)=(\la_4 b+\la_2 y)(a+x)\}| \geq \tau^{19/25}.$$
By definition, there exist at most $C'\tau^{43/25}$ pairs $((a,\la_3 a),(b,\la_4 b)) \in \mathcal A_{\la_3} \times \mathcal A_{\la_4}$ which are not good for $(\la_1,\la_2)$.

We say that a pair  $((a,\la_3 a),(b,\la_4 b)) \in \mathcal A_{\la_3} \times \mathcal A_{\la_4}$ is \textit{good} if it is good for all $(\la_1,\la_2) \in U_j \times U_j$. Repeating the previous incidence theoretic analysis for each pair $(\la_1,\la_2)$, we see that there are at most $C'N^2\tau^{43/25}$ pairs $((a,\la_3 a),(b,\la_4 b)) \in \mathcal A_{\la_3} \times \mathcal A_{\la_4}$ which are not good for some $(\la_1,\la_2)$. The other $(\frac{\tau}{4})^2 - C'N^2 \tau^{43/25}$ pairs are good.

We choose the constants $c_M$ and $c_N$ sufficiently small so as to ensure that
\begin{equation}
\left(\frac{\tau}{4}\right)^2 - C'N^2 \tau^{43/25}>\left(\frac{\tau}{4}\right)^2\left(1-\frac{1}{e(2\frac{M}{8}-3)}\right).
\label{needed}
\end{equation}
After some algebra, we see that the bound \eqref{needed} will certainly follow if
\begin{equation}
4eC'c_Mc_N^2 \leq 1
\label{req3}
\end{equation}
holds. We will verify later that our choices of $c_M$ and $c_N$ can be made in such a way as to ensure that \eqref{req3} holds.

Construct an $\frac{M}{8}$-partite graph $G$ with $\frac{M}{8}\cdot \frac{\tau}{4}$ vertices. The vertices of $G$ are the $\frac{M\tau}{32}$ elements of $P_j''$. The $M/8$ independent sets consist of the elements of $P_j''$ lying on the same line through the origin. That is, if $(a,\la a)$ and $(b,\la b)$ are elements of $\mathcal A_{\la}''$, for some $\la \in T_j''$, then we do not draw an edge connecting $(a, \la a)$ and $(b, \la b)$. So, as in the statement of Lemma \ref{EGT}, we have an $\frac{M}{8}$-partite graph, and each of the $M/8$ independent sets contains $\frac{\tau}{4}$ vertices.

The edge set for our graph $G$ is simply constructed as follows. We draw an edge between two vertices $(a, \la_3 a)$ and $(b, \la_4 b)$ from our graph if the pair $((a,\la_3 a),(b,\la_4 b))$ is good.

We can now apply Lemma \ref{EGT} with $k=\frac{\tau}{4}$ and $r=\frac{M}{8}$. It follows from \eqref{needed} that there are at more than 
$$\left(\frac{\tau}{4}\right)^2\left(1-\frac{1}{e(2\frac{M}{8}-3)}\right)$$
edges between any two of the $M/8$ independent sets. Therefore, by Lemma \ref{EGT}, we can find a copy of $K_{\frac{M}{8}}$ as an induced subgraph of $G$. Define the representative set $\mathcal B_j$ to be the set of $M/8$ vertices which make up the induced complete graph.

To complete the proof of the lemma, we just need to verify that this choice for the set $\mathcal B_j$ has the required property. Indeed, fix a quadruple $(\la_1,\la_2,\la_3,\la_4) \in U_j \times U_j \times T_j'' \times T_j''$ such that $\la_3 \neq \la_4$. Then, recall that
\begin{align*}
\mathcal E (\la_1,\la_2,\la_3,\la_4):&=|\{ r \in R((a_{\la_3},\la_3a_{\la_3})+\mathcal A_{\la_1}) \cap R((a_{\la_4},\la_4a_{\la_4})+\mathcal A_{\la_2})\}|
\\&=|\{(x,y) \in A_{\la_1} \times A_{\la_2}: (\la_3 a_{\la_3}+\la_1 x)(a_{\la_4}+y)=(\la_4 a_{\la_4}+\la_2 y)(a_{\la_3}+x)\}|.
\end{align*}
However, since the vertices $(a_{\la_3},\la_3a_{\la_3})$ and $(a_{\la_4},\la_4a_{\la_4})$ are neighbours in our graph $G$ (since they are both elements of a complete induced subgraph), it follows that they form a good pair, which in turn means that, for all $(\la_1,\la_2) \in U_j \times U_j$
$$|\{(x,y) \in A_{\la_1} \times A_{\la_2}: (\la_3 a_{\la_3}+\la_1 x)(a_{\la_4}+y)=(\la_4 a_{\la_4}+\la_2 y)(a_{\la_3}+x)\}| \leq \tau^{19/25}.$$
This completes the proof of the lemma.
\end{proof}
We also need to deal with the error term $\mathcal E (\la_1,\la_2,\la_3,\la_4)$ in the case when $\la_3=\la_4$, in which case the condition that $(\la_1,\la_3) \neq (\la_2,\la_4)$ ensures that $\la_1 \neq \la_2$. In this case we have
\begin{align}
& \mathcal E (\la_1,\la_2,\la_3,\la_3)\\&=|\{ (x,y) \in A_{\la_1}\times A_{\la_2}: R((a_{\la_3},\la_3a_{\la_3})+(x,\la_1x)) = R((a_{\la_3},\la_3a_{\la_3})+(y,\la_2 y))\}|.
\\&\leq \tau^{1-\frac{1}{25}}. \label{1eq3}
\end{align}
The inequality here follows from the fact that the representative $(a_{\la_3},\la_3 a_{\la_3})$ is chosen from the set $P_j'' \subset P_j$, and the bound \eqref{wanted} holds for all $(a,\la a) \in P_j$, provided that $\la_1 \neq \la_2$.

We will now show that Lemma \ref{thm:hardwork} and \eqref{1eq3} can be used to prove Theorem \ref{thm:main}. Plugging the bound \eqref{Ebound} and \eqref{1eq3} into \eqref{key} yields
\begin{equation}
|R_j| \geq \frac{MN\tau}{8} - M^2N^2\tau^{19/25} - MN^2\tau^{1-\frac{1}{25}}.
\label{Mbound}
\end{equation}

First of all, we choose $N$ in such a way as to ensure that the third term here is small. By choosing $c_N \leq \frac{1}{16}$, we ensure that $N \leq \frac{\tau^{1/25}}{16}$ and hence
$$MN^2\tau^{1-\frac{1}{25}} \leq \frac{MN\tau}{16}.$$
Combining this with \eqref{Mbound}, we have
\begin{equation}
|R_j| \geq \frac{MN\tau}{16} - M^2N^2\tau^{19/25}.
\label{Mbound2}
\end{equation}
We now choose $M$ in such a way as to ensure that the second term here is small. By choosing $c_M \leq \frac{1}{2}$, we ensure that $M \leq \frac{\tau^{1/5}}{2}$. We then have
$$MN \leq M\frac{\tau^{1/25}}{16} \leq \frac{\tau^{6/25}}{32},$$
from which it follows that $M^2N^2 \tau^{19/25} \leq \frac{MN\tau}{32}$, and indeed
\begin{equation}
|R_j| \geq \frac{MN\tau}{32}.
\label{Mbound3}
\end{equation}





Plugging this bound into \eqref{plan} and then applying \eqref{Sbound}, we obtain
\begin{align*}
\left| \frac{A+A}{A+A} \right| &\geq \left\lfloor \frac{|S|}{M+N} \right\rfloor \left(\frac{MN \tau}{32} \right)
\\& \geq  \left\lfloor \frac{|S|}{2M} \right\rfloor \left(\frac{MN \tau}{32} \right)
\\&\gg |S|\tau N
\\&\gg \frac{|A|^2}{\log|A|} N
\\&\gg \frac{|A|^2}{\log |A|} \tau^{1/25}.
\end{align*} 
Note that the second, third and fifth inequalities here make use of the assumption that \eqref{req1} holds. Then, after inserting the bound $\tau \gg \frac{|A|^2}{|A:A|}$, we have
$$\left| \frac{A+A}{A+A} \right| \gg \frac{|A|^{2+\frac{2}{25}}}{|A:A|^{\frac{1}{25}}\log |A|},$$
as required. 

It remains to verify that it is possible to choose the constants $c_M$ and $c_N$ in a suitable way so that these values satisfy all of the necessary conditions assumed during the argument. For the convenience of the reader, we recap these conditions below.
\begin{align}
\label{1}c_M &\leq \frac{1}{2},\\
\label{2}c_N &\leq \frac{1}{16}, \\
\label{3}4eC'c_Mc_N^2 &\leq 1, \\
\label{4} c_N &\leq \left( \frac{c_M}{4C}\right)^{1/2}, \\
\label{5} c_M\tau^{1/5} &\geq c_N \tau^{1/25}, \\
\label{7} c_N \tau^{1/25} &\geq 1, \\
\label{8} c_M\tau^{1/5} &\leq \frac{|S|}{2}.
\end{align}

Set $c_M=\min \{ \frac{1}{2}, \frac{1}{4eC'} \}$ and $c_N=\min \{ \frac{1}{16}, \left(\frac{c_M}{4C}\right)^{1/2}, c_M \}$. This ensures that \eqref{1}, \eqref{2}, \eqref{3}, \eqref{4} and \eqref{5} hold. So, it suffices to check that \eqref{7} and \eqref{8} hold.

If \eqref{7} does not hold then $\tau^{1/25}\ll 1$. It then follows from \eqref{Sbound} that $|A:A|\geq |S| \gg \frac{|A|^2}{\log |A|}$. Applying \eqref{antal}, we have
$$\left|\frac{A+A}{A+A}\right| \geq |A|^2 =\frac{|A|^{2+\frac{2}{25}}}{|A|^{\frac{2}{25}}} \gg \frac{|A|^{2+\frac{2}{25}}}{|A:A|^{\frac{1}{25}}\log^{\frac{1}{25}}|A|}.$$
In particular, the theorem holds, and so we may assume that \eqref{7} holds.

Finally, if $c_M \tau^{1/5}>\frac{|S|}{2}$ then it follows from bound \eqref{Sbound} that $\tau^{1/5} \gg \frac{|A|^2}{(\tau \log|A|)}$. However, we also have $\tau \leq |A|$, and thus $\log |A| \gg |A|^{4/5}$. This is a contradiction for sufficiently large sets $A$. \qedsymbol

\section*{Acknowledgements}The author was partially supported by the Austrian Science Fund (FWF): Project F5511-N26, which is part of the Special Research Program ``Quasi-Monte Carlo Methods: Theory and Applications". I am grateful to Abdul Basit, Herbert Fleischner, Brandon Hanson, Ben Lund, Melvyn Newton, Orit Raz, Ilya Shkredov, Arne Winterhof and Dmitry Zhelezov for various helpful conversations and corrections.

I am especially grateful to Noga Alon for his generous help in providing the proof of Lemma \ref{EGT}, which resulted in an improvement to the main result of this paper.

\appendix

\section*{Appendix: Proof that $\mathcal L$ satisfies the conditions of Lemma \ref{thm:PS}}

Firstly, we will check that two distinct curves from $\mathcal L$ intersect in at most two points.

Let $l_{(a,\la a)}$ and $l_{(b,\la' b)}$ be two distinct curves from $\mathcal L$. Their intersection is the set of all $(x,y)$ such that
\begin{align}
\la ay+\la_1 x(a+y) &=\la ax +\la_2 y(a+x) \label{boring1}
\\ \la' by+\la_1 x(b+y) &=\la' bx +\la_2 y(b+x). \label{boring2}
\end{align}
This can be rearranged into the form
\begin{align}
x \big \{ y(\la_1-\la_2) +a(\la_1-\la) \big \}&=ay(\la_2-\la) \label{boring7}
\\x \big \{ y(\la_1-\la_2) +b(\la_1-\la') \big \}&=by(\la_2-\la'). \label{boring8}
\end{align}
Let us assume that $ y(\la_1- \la_2)+a(\la_1-\la)$ and $ y(\la_1- \la_2)+b(\la_1-\la')$ are non-zero. We will deal with the case when one or more of these quantities equals zero later. With this assumption, we can conclude that
\begin{equation}
\frac{by(\la_2-\la')}{y(\la_1-\la_2) + b(\la_1-\la')}=x=\frac{ay(\la_2-\la)}{y(\la_1-\la_2) + a(\la_1-\la)}.
\label{xeq}
\end{equation}
This gives the following quadratic equation with $y$ as its variable:
$$y^2(a(\la_1-\la_2)(\la_2-\la)-b(\la_1-\la_2)(\la_2-\la')) + y ( ab (\la_1-\la')(\la_2-\la) -ab(\la_1-\la)(\la_2-\la'))=0,$$
which can be rewritten as
\begin{equation}
(\la_1-\la_2)y \big \{ y(a(\la_2-\la)-b(\la_2 - \la'))+ab(\la'-\la) \big \}=0.
\label{quad1}
\end{equation}
Recall from the definition of $\mathcal L$ that $\la_1 \neq \la_2$. Therefore, the solutions to \eqref{quad1} are $y=0$ and
\begin{equation}
y(a(\la_2-\la)-b(\la_2 - \la'))+ab(\la'-\la)=0.
\label{boring5}
\end{equation}
The latter equation gives  unique solution for $y$, unless both
\begin{align}
ab(\la'-\la)&=0 \label{boring3} \,\,\,\,\,\,\ \text{and}
\\(a(\la_2-\la)-b(\la_2 - \la'))&=0 \label{boring4}
\end{align}
hold. However, \eqref{boring3} implies that $\la= \la'$, and after feeding this information into \eqref{boring4}, it follows that $a=b$. This contradicts the assumption that the curves $l_{(a,\la a)}$ and $l_{(b,\la' b)}$ are distinct, and therefore there is at most one solution to \eqref{boring5}. There are then at most two values of $y$ solving \eqref{quad1}. One can feed each of these into \eqref{xeq} to determine the corresponding $x$-coordinates of the two intersection points.

It remains to consider the case when either $ y(\la_1- \la_2)+a(\la_1-\la)=0$ or $ y(\la_1- \la_2)+b(\la_1-\la')=0$. We will check only the first of these possibilities; the second can be verified by a symmetric argument.

Indeed, suppose that $ y(\la_1- \la_2)+a(\la_1-\la)=0$. Since $\la_1 \neq \la_2$, we have $y=\frac{a(\la-\la_1)}{\la_1-\la_2}$. We need to check that there is no point of intersection of the two curves $l_{(a,\la a)}$ and $l_{(b, \la b)}$ with $y$ coordinate equal to $\frac{a(\la-\la_1)}{\la_1-\la_2}$. This is true, because the curve $l_{(a,\la a)}$ does not meet the line $y=\frac{a(\la-\la_1)}{\la_1-\la_2}$, as the following calculation shows.

Recalling \eqref{boring7} and our assumption that $ y(\la_1- \la_2)+a(\la_1-\la)=0$, we see that a point of intersection of the curve $l_{(a,\la a)}$ and the line $y=\frac{a(\la-\la_1)}{\la_1-\la_2}$ is only possible if
$$a^2\frac{\la-\la_1}{\la_1-\la_2}(\la_2-\la)=0,$$
which is a contradiction.


The next task is to check that, given any two points from $\mathcal P$, there exist at most two curves which pass through both.

This is straightforward in this case. All of the curves intersect at the origin. The previous argument shows that they intersect in at most one other point. Since the origin is not an element of $\mathcal P$ (this is because of the original assumption that $A$ consists only of strictly positive elements), it follows that any two curves from $\mathcal L$ contain at most one point from $\mathcal P$ in their intersection. It then follows immediately that there is at most one curve from $\mathcal L$ which passes through two points $p,q \in \mathcal P$. This is because, if there were two distinct curves passing through $p$ and $q$, then their intersection would contain at least two points from $\mathcal P$, which is impossible. \qedsymbol

\section*{Appendix: Proof that $\mathcal L'$ satisfies the conditions of Lemma \ref{thm:PS}}

Firstly, we will check that two distinct curves from $\mathcal L'$ intersect in at most two points.

Let $l_{(a,b)}$ and $l_{(a',b')}$ be two distinct curves from $\mathcal L'$. Their intersection is the set of all $(x,y)$ such that
\begin{align}
(\la_3 a+\la_1x) (b+y) &=(\la_4 b +\la_2 y )(a+x) \label{boring11}
\\ (\la_3 a'+\la_1x) (b'+y) &=(\la_4 b' +\la_2 y )(a'+x). \label{boring12}
\end{align}
This can be rearranged into the form
\begin{align}
x \big \{ y(\la_1-\la_2) +b(\la_1-\la_4) \big \}&=a\big \{y(\la_2-\la_3)+b(\la_4-\la_3) \big \} \label{boring17}
\\x \big \{ y(\la_1-\la_2) +b'(\la_1-\la_4) \big \}&=a'\big \{y(\la_2-\la_3)+b'(\la_4-\la_3) \big \}. \label{boring18}
\end{align}
Let us assume that $ y(\la_1- \la_2)+b(\la_1-\la_4)$ and $ y(\la_1- \la_2)+b'(\la_1-\la_4)$ are non-zero. We will deal with the case when one or more of these quantities equals zero later. With this assumption, we can conclude that
\begin{equation}
\frac{a' \big \{y(\la_2-\la_3)+b'(\la_4-\la_3) \big \}}{y(\la_1-\la_2) + b'(\la_1-\la_4)}=x=\frac{a\big \{y(\la_2-\la_3)+b(\la_4-\la_3) \big \}}{y(\la_1-\la_2) + b(\la_1-\la_4)}.
\label{xeq2}
\end{equation}
This gives the following quadratic equation with $y$ as its variable
\begin{align*}
&y^2\big \{a(\la_1-\la_2)(\la_2-\la_3)-a'(\la_1-\la_2)(\la_2-\la_3) \big \} 
\\ + & y \big \{ ab' (\la_1-\la_4)(\la_2-\la_3) +ab(\la_1-\la_2)(\la_4-\la_3))-a'b(\la_2-\la_3)(\la_1-\la_4)-a'b'(\la_4-\la_3)(\la_1-\la_2) \big \}
\\ + & \big \{ abb'(\la_4-\la_3)(\la_1-\la_4) - a'bb'(\la_1-\la_4)(\la_4-\la_3) \big \}=0.
\end{align*}
There are at most two values of $y$ which give a solution to this quadratic, provided that not all of the coefficients are zero. 

Suppose that we are in this degenerate case. Recall that it follows from the definition of $\mathcal L'$ that $\la_3 \neq \la_4$. Then, since the constant coefficient of the above quadratic is zero, we have
$$abb'(\la_4-\la_3)(\la_1-\la_4) - a'bb'(\la_1-\la_4)(\la_4-\la_3)=0,$$
and it therefore follows that $a=a'$. Combining this information with our assumption that the linear coefficient is zero, it follows from some simple algebra that
$$(b'-b)(\la_1-\la_3)(\la_2-\la_4)=0,$$
and therefore $b=b'$. This contradicts our assumption that the curves $l_{a,b}$ and $l_{a',b'}$ are distinct, and thus it follows that there are at most two solutions for $y$ in the aforementioned quadratic equation. One can feed each of these into \eqref{xeq2} to determine the corresponding $x$-coordintes of the two intersection points.

It remains to consider the case when either $ y(\la_1- \la_2)+b(\la_1-\la_4)=0$ or $ y(\la_1- \la_2)+b'(\la_1-\la_4)=0$. We will check only the first of these possibilities; the second can be verified by a symmetric argument.

Indeed, suppose that $y(\la_1-\la_2)+b(\la_1-\la_4)=0$. This implies that $\la_1 \neq \la_2$ and thus we can write
$$y=b\frac{\la_4-\la_1}{\la_1-\la_2}.$$
We need to check that there is no point of intersection of the two curves $l_{a,b}$ and $l_{a',b'}$ with $y$ coordinate equal to $b\frac{\la_4-\la_1}{\la_1-\la_2}$. This is true, because the curve $l_{a,b}$ does not meet the line $y=b\frac{\la_4-\la_1}{\la_1-\la_2}$, as the following calculation shows.

Recalling \eqref{boring17} and our assumption that $ y(\la_1- \la_2)+b(\la_1-\la_4)=0$, we see that a point of intersection of the curve $l_{(a,b)}$ and the line  $y=b\frac{\la_4-\la_1}{\la_1-\la_2}$ is only possible if
$$a\left(b\frac{\la_4-\la_1}{\la_1-\la_2}(\la_2-\la_3)+b(\la_4-\la_3)\right)=0.$$
After some simple algebra, it follows that this happens if and only if $(\la_3-\la_1)(\la_2-\la_4)=0$, which is a contradiction.

The final task is to show that, for any two distinct points $p=(x_0,y_0)$ and $q=(x_1,y_1)$ from $\mathcal P$, there are at most two curves from $\mathcal L$ which pass through both $p$ and $q$.

The number of curves $l_{a,b}$ passing through both $p$ and $q$ is equal to the number of solutions $(a,b)$ to the system of equations
\begin{align*}
(\la_3 a + \la_1 x_0)(b+y_0) &= (\la_4 b + \la_2 y_0)(a+x_0)
\\ (\la_3 a + \la_1 x_1)(b+y_1) &= (\la_4 b + \la_2 y_1)(a+x_1),
\end{align*}
which can be rearranged into the form
\begin{align}
a \big \{ b(\la_3-\la_4) + y_0(\la_3-\la_2) \big \} &= b(x_0(\la_4-\la_1))+x_0y_0(\la_2-\la_1) \label{bor1}
\\ a \big \{ b(\la_3-\la_4) + y_1(\la_3-\la_2) \big \} &= b(x_1(\la_4-\la_1))+x_1y_1(\la_2-\la_1). \label{bor2}
\end{align}
Let us assume that both $b(\la_3-\la_4) + y_0(\la_3-\la_2)$ and $b(\la_3-\la_4) + y_1(\la_3-\la_2)$ are nonzero. We will deal with the case when one or more of these quantities equals zero later. We can then make $a$ the subject of each of these equations, resulting in the following quadratic with $b$ as its variable:
\begin{align*}
&b^2 \big \{ (x_0-x_1)(\la_3-\la_4)(\la_4-\la_1)\big \}
\\+&b \big \{(x_0y_0-x_1y_1)(\la_2-\la_1)(\la_3-\la_4) -(x_1y_0-x_0y_1)(\la_4-\la_1)(\la_3-\la_2) \big \}
\\+&y_0y_1(\la_3-\la_2)(\la_2-\la_1)(x_0-x_1)=0.
\end{align*}
There are, at most, two values of $b$ satisfying this equation. These solutions can then be plugged into \eqref{bor1} in order to find the two solutions $(a,b)$. The only exception occurs if all of the coefficients of the quadratic are zero. Suppose for a contradiction that this is the case. The $b^2$ coefficient being equal to zero implies that $x_0=x_1$. Then, combining this information with the assumption that the linear coefficient is zero, it follows, after some algebraic manipulation, that
$$(y_0-y_1)(\la_2-\la_4)(\la_3-\la_1)=0,$$
and thus $y_0=y_1$, contradicting the assumption that $p$ and $q$ are distinct.

It remains to consider the case when either $b(\la_3-\la_4) + y_0(\la_3-\la_2)=0$ or $ b(\la_3-\la_4) + y_1(\la_3-\la_2)=0$. We will check only the first of these possibilities; the second can be verified by a symmetric argument.

Indeed, suppose that $b(\la_3-\la_4) + y_0(\la_3-\la_2)=0$. Since $\la_3 \neq \la_4$ we can write
\begin{equation}
b=y_0\frac{\la_2-\la_3}{\la_3-\la_4}.
\label{beq}
\end{equation}
We need to check that, for this choice of $b$, there is no $a$ for which $(x_0,y_0) \in l_{a,b}$. Suppose for a contradiction that such an $a$ exists. Then, recalling \eqref{bor1}, we have
$$a \big \{ b(\la_3-\la_4) + y_0(\la_3-\la_2) \big \} = b(x_0(\la_4-\la_1))+x_0y_0(\la_2-\la_1).$$
For such an $a$ to exist as assumed, it must be the case that the right hand side equals zero. Plugging in \eqref{beq}, we obtain
$$x_0y_0\frac{(\la_2-\la_3)(\la_4-\la_1)}{\la_3-\la_4}+x_0y_0(\la_2-\la_1)=0.$$
This implies that
$$(\la_1-\la_3)(\la_4-\la_2)=0,$$
which is a contradiction.


\end{document}